\documentclass[12pt]{amsart}
\usepackage[linesnumbered, ruled, vlined]{algorithm2e}
\usepackage{ fullpage, amsmath, amsthm, amssymb, enumerate, amsfonts, xypic, hyperref}

\SetKwInOut{Input}{Input}
\SetKwInOut{Output}{Output}

\def\ZZ{{\mathbb Z}}
\def\NN{{\mathbb N}}

\def\CC{{\mathbb C}}

\def\RR{{\mathbb R}}
\def\QQ{{\mathbb Q}}

\def\OO{{\mathcal O}}

\def\JR{{\text{\textnormal{JR}}}}

\newtheorem{theorem}{Theorem}[section]
\newtheorem*{theorem*}{Main Theorem}

\newtheorem{problem}[theorem]{Problem}
\newtheorem{lemma}[theorem]{Lemma}
\newtheorem{prop}[theorem]{Proposition}

\author{Caleb Springer}
\address{Department of Mathematics, The Pennsylvania State University, University Park, PA 16802, USA}
\email{cks5320@psu.edu}
\title{Undecidability, unit groups, and some totally imaginary infinite extensions of $\QQ$}
\date{\today}

\begin{document}
\begin{abstract}
We produce new examples of totally imaginary infinite extensions of $\QQ$ which have undecidable first-order theory by generalizing the methods used by Mart\'{\i}nez-Ranero, Utreras and Videla for $\QQ^{(2)}$. In particular, we use parametrized families of polynomials whose roots are totally real units to apply methods originally developed to prove the undecidability of totally real fields.  This proves the undecidability of $\QQ^{(d)}_{ab}$ for all $d \geq 2$.
\end{abstract}
\maketitle

%%%
  %
  %%%
\section{Introduction}
  There is a long history to the decision problem for theories, stemming from the work of Church, Hilbert and Turing, among many others. 
Notably, Church showed that there is no consistent decidable extension of Peano arithmetic, and proved, along with Turing, the negative answer to Hilbert's \emph{Entscheidungsproblem}; see \cite[\S4.3]{wolf} for a historical overview.

Some of the most famous questions and results in decidability have arisen from number theory.  
For example, Hilbert's Tenth Problem asks whether the positive-existential theory of $\ZZ$ is decidable, or as it was originally formulated, whether there is an algorithm which decides, given a multivariable polynomial equation with integer coefficients, whether or not the equation has an integer solution.
Famously, Matiyasevich proved that this problem is undecidable \cite{mat} by building on the work of Davis, Putnam and J. Robinson \cite{dpr}. 
The same problem with $\ZZ$ replaced by $\QQ$ is known as Hilbert's Tenth Problem over $\QQ$, and remains one of the biggest open problems in the intersection of number theory and decidability.

While Hilbert's Tenth Problem over $\QQ$ currently remains unknown, the decidability of the entire first-order theory of $\QQ$, or more generally an algebraic extension of $\QQ$, is another interesting question in its own right. Julia Robinson proved that $\QQ$ has undecidable first-order theory \cite{jr-Q}, and extended her result to every number field \cite{jr-F}.  Rumely  generalized this result further and proved that every global field is undecidable \cite{rumely}.
    However, the following question remains.
  
\begin{problem}
	Which infinite algebraic extensions of $\QQ$ have undecidable first-order theory?
\end{problem} 

There are many results which partially answer this question and demonstrate that some infinite algebraic extensions of $\QQ$ have undecidable first-order theory, while others do not.  For example, Shlapentokh proved that every abelian extension of $\QQ$ which is ramified at only finitely many primes is undecidable \cite[Theorem 5.5]{shlap}, generalizing a result of Videla \cite{vid-cycl}.  On the other hand, recall that an algebraic number is said to be totally real if its minimal polynomial has only real roots. Fried, Haran and V\"olklein proved that the field $\QQ^{\text{tr}}$ of all totally real numbers is decidable \cite{fhv}. 

One way to prove that a field $K$ is undecidable is to show that the ring of integers $\OO_K$ is undecidable and definable in $K$.  For example, this method proves that the field $K = \QQ( \{\sqrt p : p \text{ prime}\})$ is undecidable.  In this case, the definability of $\OO_K$ in $K$ was shown by Videla \cite[\S5.4]{vid}, while the undecidability of $\OO_K$ was proven by Julia Robinson \cite{jr} and follows from a general ``blueprint" that she described.  The blueprint, in a more general form due to C.W. Henson \cite[\S3.3]{vdd}, states that a ring of integers is undecidable if there is a definable family of subsets which contains finite sets of arbitrarily large cardinality.   As an application of this blueprint, Julia Robinson showed how such families of sets could be constructed for rings of totally real numbers by using totally positive elements.
Vidaux and Videla expanded on her ideas to prove the undecidability of a large class of rings of integers $\OO_K$ in totally real fields $K$  \cite{vv-nested, vv-northcott}.  When combined with definability results, such as those of Fukuzaki \cite{fuku}, Shlapentokh \cite{shlap}, and Videla \cite{vid}, this proves the undecidability of many totally real fields.

Recently, Mart\'{\i}nez-Ranero, Utreras and Videla \cite{mruv} leveraged these methods, which were  developed for totally real fields, to instead prove the undecidability of the totally imaginary field $\QQ^{(2)}$, the compositum of all degree 2 extensions of $\QQ$.  
The key to their strategy was noticing that $\QQ^{(2)}$ is a totally imaginary quadratic extension of the totally real field $K = \QQ(\{\sqrt p : p \text{ prime}\})$, and therefore $\OO_K^\times$ is a finite-index subgroup of $\OO_{\QQ^{(2)}}^\times$. 
Using this fact, they produce a special ``large" set $W$, which is definable in $\OO_{\QQ^{(2)}}$ and contains only totally real elements. The undecidability of $\QQ^{(2)}$ then follows from the aforementioned methods for totally real rings of integers and the definability of $\OO_{\QQ^{(2)}}$ in ${\QQ^{(2)}}$.
The goal of this paper is to generalize their strategy and produce more examples of totally imaginary infinite extensions of $\QQ$ with undecidable first-order theory.

 Given a number field $F$, let $F^{(d)}$ be the compositum of all extensions of $F$ of degree at most $d$, and let $F^{(d)}_{ab}$ be the maximal abelian subextension of $F\subseteq F^{(d)}$.  The following theorem will be proved as Theorem \ref{full-thm}, and the undecidability of $\QQ^{(2)}$ follows as a special case.  In Section \ref{examples}, we give additional explicit examples of totally real fields $K$ and families of polynomials $\{f_a(x) \mid a\in \ZZ_{\geq N_0}\}$ which satisfy the conditions of the theorem.

  \begin{theorem*} 
 Let $K$ be an infinite totally real extension of $\QQ$ which is contained in $F^{(d)}_{ab}$ for some $d\geq 2$ and some number field $F$.  Assume $K$ contains all roots of a parametrized family of polynomials 
  $$
	\{f_a(x) = x^n + p_{n-1}(a)x^{n-1} + \dots p_1(a)x +p_0(a) \mid a\in \ZZ_{> N_0}\}
$$
where  each $p_i(t)\in \ZZ[t]$ is a polynomial, $p_0(t) = \pm1$ is constant and $p_j(t)$ is nonconstant for some $1\leq j\leq n-1$.  If $L$ is any  totally imaginary quadratic extension of $K$, then the first-order theory of  $L$ is undecidable.
   \end{theorem*}

In this theorem, the ring of integers $\OO_K$ is undecidable by a result of Vidaux and Videla \cite{vv-northcott}.  This fact is used to deduce the undecidability of $\OO_L$, and therefore the undecidability of $L$ because $\OO_L$ is definable in $L$ by a result of Shlapentokh \cite{shlap}.  The strategy, as in the case of $\QQ^{(2)}$, is to exploit the unit group $\OO_L^\times$ to define a sufficiently large subset $W$ of $\OO_L$ which contains only totally real elements.  This is done by explicitly using a polynomial whose values are power-sums of the units defined by the family $\{f_a(x)\}$.
Producing this special polynomial is the main ingredient required to extend the method for $\QQ^{(2)}$ to this more general setting.
 It would be interesting to investigate whether the the undecidability of a totally real field $K$ implies the undecidability of its totally imaginary quadratic extensions in a more general context, without the need for the explicit parametrized family of polynomials or the field $F^{(d)}_{ab}$ containing $K$.

\subsection*{Acknowledgements}
The author would like to thank Hector Pasten for his helpful comments, and for suggesting a shorter proof of Lemma \ref{poly-lem}.  The author also thanks Kirsten Eisentr\"ager, Arno Fehm, Alexandra Shlapentokh and Xavier Vidaux for all of their comments and suggestions.  The author was partially supported by National Science Foundation award CNS-1617802.

%%%
%
%%%
\section{Sufficient Conditions For Undecidability}
Throughout this paper, $K$ will denote a totally real infinite extension of $\QQ$, and $L$ a totally imaginary quadratic extension of $K$.
The goal of this section is to find a suitable sufficient condition for when $L$ has undecidable first-order theory.  We will begin by reviewing some definability results and recall the methods used to prove undecidability in the totally real case.  

Clearly, if $\OO_L$ is definable in $L$ and the first-order theory of $\OO_L$ is undecidable, then the first-order theory of $L$ is also undecidable. There are many results on the definability of rings of integers in infinite algebraic extensions of $\QQ$; see, for example, the work of Fukuzaki \cite{fuku}, Shlapentokh \cite{shlap}, and Videla \cite{vid}.  The following result of Shlapentokh, presented in \cite[Example 4.3]{shlap}, will suffice for our purposes in this paper.

\begin{theorem}
\label{definability}
If $L$ is the compositum of finite extensions of $\QQ$ of degree less than some positive integer $d$, then $\OO_L$ is first-order definable in $L$.
 \end{theorem}
 
Now our problem reduces to finding undecidable rings of integers.  The following condition, first presented by Julia Robinson \cite[Theorem 2]{jr} and later generalized by C.W. Henson \cite[\S3.3]{vdd}, gives a sufficient condition for when a ring of integers has undecidable first-order theory.

\begin{lemma} 
\label{henson-lem}
Let $\OO$ be a ring of algebraic integers. If there is a family $\mathcal F$ of subsets of $\OO$, parametrized by an $\mathcal L_{ring}$-formula, which contains finite sets of arbitrarily large cardinality, then $\OO$ has undecidable first-order theory.
 \end{lemma}
 
 %%%
%
%%%
 \subsection{Totally real fields}\label{totally-real-section}
  
For totally real rings of integers, Lemma \ref{henson-lem} enables a concrete method, first used by Julia Robinson \cite{jr} and developed further by Vidaux and Videla \cite{vv-nested, vv-northcott}, to prove undecidability.
Given a set $X$ of totally real algebraic numbers, define
$$
	X_t = \{\alpha\in X : 0\ll \alpha \ll t\}
$$
where $0\ll \alpha \ll t$ means that every conjugate of $\alpha$ lies in the interval $(0,t)$.
The \emph{JR-number} of $X$ is 
$$
	\JR(X) = \inf\{t\in \RR : \#X_t = \infty \}.
$$
When $X = \OO_K$ for a totally real field $K$, the sets $X_t$ are definable in $X$ because every totally positive algebraic number is the sum of four squares by a theorem of Siegel \cite{siegel}, and therefore the following theorem follows from Lemma \ref{henson-lem}.

\begin{theorem}[\cite{jr}]
\label{jr-number}
If $K$ is a totally real field and the JR-number $\JR(\OO_K)$ is either a minimum or infinite, then the first-order theory of $\OO_K$ is undecidable.
 \end{theorem}
 
 Examples where this theorem applies include the ring of all totally real algebraic integers $\ZZ^{\text{tr}}$, and the ring of integers of $\QQ(\{\sqrt{p} : p \text{ prime}\})$.  The JR-numbers are 4 and $\infty$, respectively, in these cases \cite{jr}.   For many years, there were no known examples of rings of totally real integers whose JR-numbers were finite and either different from 4, or not a minimum.  Recently, infinitely many such examples have been constructed by Castillo Fernandez, Vidaux and Videla \cite{castillo-thesis, cfvv, vv-nested}, and by Gillibert and Ranieri \cite{gr}.
For the purposes of this paper, we will only be concerned with totally real fields whose JR-number is infinite.
 
 More examples of totally real fields $K$ with $\JR(\OO_K) = \infty$ come from a connection to the Northcott property, discovered by Vidaux and Videla \cite{vv-northcott}. A set $X$ is said to have the \emph{Northcott property} if
$$
	\{\alpha \in X : \hat{h}(\alpha) < t\}
$$	
is a finite set for every positive real number $t$, where $\hat{h}(\alpha)$ denotes the logarithmic Weil height.

\begin{prop}[{\cite[Theorem 2]{vv-northcott}}]
\label{north-jr}
If a totally real field $K$ has the Northcott property, then $\JR(\OO_K) = \infty$.
 \end{prop}

The following theorem of Bombieri and Zannier provides many examples of fields with the Northcott property; see also \cite{cw,wid} for more examples.  By definition, it suffices to show that the totally real field $K$ is contained in a (possibly imaginary) field which has the Northcott property.

\begin{theorem}[{\cite[Theorem 1]{bz}}]
\label{northcott}
If $F$ is a number field, then $F^{(d)}_{ab}$ has the Northcott property.
 \end{theorem}

%%%
%
%%%
\subsection{Moving to totally imaginary fields}
Recall that the field $\QQ^{(2)}$ is a totally imaginary quadratic extension of the totally real field $K = \QQ(\{\sqrt{p} : p \text{ prime}\})$.
The proof of the undecidability of $\QQ^{(2)}$ given by Mart\'{\i}nez-Ranero, Utreras and Videla \cite{mruv} uses the fact that $\JR(\OO_K) =\infty$ to deduce that $\OO_{\QQ^{(2)}}$ is also undecidable.  Essentially, they find a special set which is definable in $\OO_{\QQ^{(2)}}$ and contains only totally real elements, which allows them to apply some of the methods created for the totally real case.  
Below, we make this explicit by generalizing their strategy to a set of lemmas which provide a sufficient condition for undecidability.
  
  \begin{lemma} 
  \label{key-lem}
 If there is a first-order definable subset $W\subseteq \OO_L$ such that $\NN \subseteq W\subseteq \OO_K$, then the first-order theory of $\OO_L$ is undecidable.
   \end{lemma}
    \begin{proof} 
    This result follows quickly from Lemma \ref{henson-lem}.
   Using $W$, define a family $\mathcal F$ of subsets of  $\OO_L$ parametrized by the formula $\phi_W(x;a,b)$:
 $$
 	ax\neq 0 
		\wedge ax \neq b 
		\wedge \exists x_1,\dots, x_8\in W
		[ax = x_1^2 + \dots + x_4^2 \wedge (b - ax) = x_5^2 + \dots + x_8^2]
 $$
 If $a,b\in \NN$, then $\phi_W(x;a,b)$ implies $0\ll ax\ll b$, and $\phi_W(n;a,b)$ holds for every natural number $0 <n< \frac{b}{a}$ by Lagrange's four square theorem.  Therefore, we have
 \begin{equation}
 \label{containment-eq}
 	\left\{n\in \NN : 0 < n < \frac{b}{a}\right\} 
		\subseteq \{x\in \OO_L : \phi_W(x;a,b)\}
		\subseteq \left\{x\in \OO_K : 0\ll x\ll \frac{b}{a}\right\}.
 \end{equation}
 for every $\frac{b}{a}\in \QQ_{>0}$.  Because $\JR(\OO_K) = \infty$, the sets on the righthand side above are finite for all $\frac{b}{a}\in \QQ_{>0}$ by definition.  Hence, the family of sets $\mathcal F$ contains finite sets of arbitrarily large size as $\frac{b}{a}\to\infty$, and therefore Lemma \ref{henson-lem} applies.
  \end{proof}

  The proof above demonstrates why we restrict to the case of $\JR(\OO_K) = \infty$.  If we instead have that $\JR(\OO_K) < \infty$ is a minimum, then we could attempt to perform a similar trick. However, the sets on the right-hand side of \eqref{containment-eq} are only finite for $\frac{b}{a} < \JR(\OO_K) < \infty$ in this case.  Thus we would need to replace the lower bound given by $\NN$ on the left-hand side of \eqref{containment-eq} with a larger lower bound to obtain our finite sets of arbitrarily large cardinality. This corresponds to modifying the hypothesis of the lemma to require $S\subseteq W\subseteq \OO_K$ for some subset $S$ such that $\#\{x\in S : \phi_S(x;a,b)\}\to \infty$ as $\frac{b}{a}$ approaches $\JR(\OO_K)$ from below.
There are no obvious candidates for a convenient and useful choice of $S$, due in part to the difficulties involved in constructing examples of totally real fields $K$ such that $\JR(\OO_K)$ is finite.  Hence we will only consider the case $\JR(\OO_K) = \infty$.

The upshot of the previous lemma is that we simply need to produce the desired set $W$. 
In the case of $\QQ^{(2)}$, Mart\'{\i}nez-Ranero, Utreras and Videla use a discrete derivative trick and Hilbert's solution to Waring's problem to produce the set $W$; see \cite[Lemma 7]{mruv}.  Their strategy also works in our more general setting, but we will instead use a theorem of Kamke which solves a more general conjecture of Waring.  In the following section, we will use families of polynomials and the unit group $\OO_L^\times$ to produce the polynomial $f(x)$ and definable set $X_0$ required by the lemma.

\begin{lemma} 
  \label{poly-lem}
Let $f\in \ZZ[x]$ be a nonconstant polynomial,
 and let 
$X_0\subseteq \OO_K$ be a subset which is definable in $\OO_L$.
  If $f(n)\in X_0$ for each sufficiently large natural number $n \geq N_0$, then there is a first-order definable subset $W$ of $\OO_L$ such that $\NN \subseteq W \subseteq \OO_K$.
 \end{lemma}
 
  \begin{proof} 

Define $X_1 = \{\pm x : x\in X_0\}$.
By replacing $f(x)$ with $\pm f(x + k)$ for some $k \geq N_0$,
we can assume that $f(n)\in X_1$ is a nonnegative integer for all integers $n \geq 0$ without loss of generality.
 Then Kamke's theorem \cite{kamke} states that there is an integer $r\geq 1$ such that every $m\in \NN$ can be written in the form
$$
	m = f(a_1) + \dots + f(a_{s_1}) + s_2
$$
where $s_1,s_2\in \NN$ satisfy $s_1 + s_2\leq r$ and $a_1,\dots, a_{s_1}\in \NN$.

Thus, we may simply define
$$
	W = \bigcup_{s_1=0 }^r \bigcup_{s_2=0 }^{r-s_1} \{x_1 + \dots + x_{s_1} + s_2 : x_1,\dots, x_{s_1}\in X_1\}.
$$
  \end{proof}

  %%%
  %
  %%%
  \section{Using the group of units}
 To complete the proof of our main theorem, we will exploit the structure of the group of units $\OO_L^\times$ to apply Lemma \ref{poly-lem}. We begin by recalling two basic facts, which generalize \cite[Lemmas 5-6]{mruv}.   Throughout this section, we assume that $\JR(\OO_K) = \infty$.

\begin{lemma} 
	The group of roots of unity $\mu_L \subset \OO_L^\times$ is finite.
 \end{lemma}
  \begin{proof} 
If $\omega \in \mu_L$ is a root of unity, then $2 + \omega + \omega^{-1} \in K$ satisfies 
$
	0 \ll 2 + \omega + \omega^{-1} \ll 4.
$
Because $\JR(\OO_K) >4$, there are only finitely many elements $\alpha\in K$ satisfying $0\ll \alpha\ll 4$.
  \end{proof}

\begin{lemma}
\label{cm-lem}
	Write $\#\mu_L = 2N$.  If $u\in \OO_L^\times$, then $u^{2N}\in \OO_K^\times$.
 \end{lemma}
  \begin{proof} 
  Use the following notation.  Let $K = \cup_{i = 0}^\infty K_i$ where $K_0 \subseteq K_1\subseteq \dots$ is an infinite tower of totally real number fields such that $L_0$ is a totally imaginary quadratic extension of $K_0$, $L_n = K_n L_0$ and $L = KL_0$.
The lemma then immediately follows from the fact that $[\OO_{L_n}^\times : \mu_{L_n} \OO_{K_n}^\times]\in \{1,2\}$ for all $n\geq 0$ \cite[Theorem 4.12]{w}.
  \end{proof}
  
  The previous lemma implies $(\OO_L^\times)^{2N}\subseteq \OO_K$ is definable in $\OO_L$, which will allow us to produce a subset $X_0\subseteq \OO_K$ which is definable in $\OO_L$.
   Next, we will define a useful multivariable polynomial, then specialize it to a certain single-variable polynomial $f(x)$ which satisfies Lemma \ref{poly-lem}.
   
We will use the following notation.  For each $k,n\geq 1$, write $q_k(x_1,\dots, x_{n}) = x_1^k + \dots + x_{n}^k$ for the $k$-th power-sum polynomial, and let $s_k$ be the $k$-th elementary symmetric polynomial.  The \emph{Newton-Girard formulae} provide the following relation for any $k,n\geq 1$.  For ease of notation, we suppress the variables on the right-hand side.
   \begin{align*}
		q_k(x_1,\dots, x_{n})
			&= (-1)^{k-1} ks_k +\sum_{i = 1}^{k-1} (-1)^{k+i-1}s_{k-i} q_{i}.
   \end{align*}
  
  \begin{lemma} 
  \label{poly-sum}
Given any integers $m, n \geq 1$, there is a polynomial 
$$
	Q_m(x_0, \dots, x_{n-1})\in \ZZ[x_0, \dots, x_{n-1}]
$$
 such that for any $(a_0,\dots, a_{n-1})\in \ZZ^n$,
  $$
  	Q_m(a_0, \dots, a_{n-1}) = \alpha_1^m + \dots + \alpha_n^m
  $$
   where $\alpha_1,\dots, \alpha_n$ are the roots of $f(x) = x^n + a_{n-1}x^{n-1} + \dots + a_0\in \CC[x]$.
   \end{lemma}
    \begin{proof} 
   For any polynomial $x^n + a_{n-1}x^{n-1} + \dots + a_0 \in \CC[x]$, the roots $\alpha_1,\dots, \alpha_n\in \CC$ satisfy
   \begin{align*}
   	s_1(\alpha_1,\dots, \alpha_n) &= \alpha_1 +\dots + \alpha_n = -a_{n-1}\\
	\vdots\\ 
	s_n(\alpha_1, \dots, \alpha_n) &= \alpha_1\dots\alpha_n = (-1)^n a_0.
   \end{align*}  
   
   Therefore, by using the Newton-Girard formulae and induction, we can write $q_m(\alpha_1,\dots, \alpha_n)$ as a polynomial in $a_{n-k} = (-1)^{k}s_k(\alpha_1, \dots, \alpha_n)$ for $1\leq k\leq n$, as claimed.
    \end{proof}

We focus our attention on the values that $Q_m$ takes on the coefficients of particular families of polynomials whose roots are totally real units.  Importantly, we need the resulting single-variable polynomial to be nonconstant to apply Lemma \ref{poly-lem}.

  \begin{lemma} 
  Let $p_0(t), \dots, p_{n-1}(t)\in \ZZ[t]$ be polynomials which parametrize a family of polynomials 
  $$
	\{f_a(x) = x^n + p_{n-1}(a)x^{n-1} + \dots p_1(a)x +p_0(a) : a\in \ZZ_{\geq N_0}\}
$$
where $p_{j}(x)$ is nonconstant for some $0\leq j\leq n-1$.  For any $N \geq 1$, there is some $k\geq 1$ such that 
$$
	Q_{kN}(p_0(x), \dots, p_{n-1}(x))
$$
is nonconstant.
\label{nonconst-lem} 
   \end{lemma}
    \begin{proof} 
Factor each polynomial $f_a(x) = \prod_{i = 1}^n (x - \alpha_{i,a})$ over the algebraic closure. First consider the case of $N = 1$.    By assumption, there is a smallest index $1\leq j_0\leq n$ such that $s_{j_0}(\alpha_{1,a},\dots, \alpha_{n,a})= (-1)^{j_0}p_{n-j_0}(a)$ is nonconstant as $a$ varies.  For each $1\leq k\leq n$, the Newton-Girard formulae
$$
	q_{k}
		 =  (-1)^{k-1} k s_{k} +\sum_{i = 1}^{k-1} (-1)^{k+i-1}s_{k-i} q_{i}.
$$
can be expanded recursively to write $q_k$ in terms of $s_1, \dots, s_{k}$. This
implies that 
$$
	Q_{k}(p_0(a), \dots, p_{n-1}(a)) = q_k(\alpha_{1,a},\dots, \alpha_{n,a})
$$
 is constant for $1\leq k < j_0-1$, and nonconstant for $k = j_0$.

Now let $N > 1$. We will reduce to the previous case by defining a new family of polynomials  $\hat{f}_a(x) = \prod_{i = 1}^n (x - \alpha_{i,a}^N)$.  First, we show that the new family $\{\hat{f}_a(x) : a\in \ZZ_{\geq N_0}\}$ is also parametrized.
For each $1\leq j\leq n$, the $(n-j)$-th coefficient of $\hat{f}_a(x)$ is equal to $(-1)^j s_j(\alpha_1^N, \dots, \alpha_n^N)$.
 Clearly the polynomial $s_j(x_1^N,\dots, x_n^N)$ is invariant under permutation of variables, so there is a polynomial $g_j(t_1,\dots, t_n)\in \ZZ[t_1,\dots, t_n]$ such that 
   $$
   	s_j(x_1^N,\dots, x_n^N) = g_j(s_1(x_1,\dots, x_n), \dots, s_n(x_1,\dots, x_n));
$$
see \cite[Theorem IV.6.1]{lang}.
Therefore,
$$
	\hat{f}_a(x) = x^n + \hat{p}_{n-1}(a) x^{n-1} + \dots + \hat{p}_0(a)
$$
where $\hat{p}_{n-j}(x) = (-1)^{j} g_{j}(-p_{n-1}(x),\dots, (-1)^np_0(x))$ for each $1\leq j \leq n$. 

This shows that $\{\hat{f}_a(x) : a\in \ZZ_{\geq N_0}\}$ is a parametrized family of polynomials, so it only remains to see that some $\hat{p}_j(x)$ is nonconstant.  If $\hat{p}_j(x)$ is constant for all $1\leq j\leq n-1$, then the family $\{\hat{f}_a(x): a\in \ZZ_{\geq N_0}\}$  contains a single polynomial.  But this is clearly impossible by the definition of $\hat{f}_a(x)$ because the family $\{{f}_a(x): a\in \ZZ_{\geq N_0}\}$ is infinite by assumption.  Hence applying the base case to the family $\{\hat{f}_a(x) : a\in \ZZ_{\geq N_0}\}$ completes the proof because
$$
	Q_{kN}(p_0(a), \dots, p_{n-1}(a))
	=\alpha_{1,a}^{kN} + \dots +  \alpha_{n,a}^{kN}
	=Q_{k}(\hat{p}_0(a), \dots, \hat{p}_{n-1}(a))
$$
for all $k\geq 1$ by construction.
    \end{proof}

  We are now ready to prove the main theorem of this section on the undecidability of rings of integers in totally imaginary fields.

  \begin{theorem} 
  \label{main-thm}
Let $K$ be a totally real field with $\JR(\OO_K) = \infty$, and let $p_0(t), \dots, p_{n-1}(t)\in \ZZ[t]$ be polynomials which parametrize a family of polynomials 
  $$
	\{f_a(x) = x^n + p_{n-1}(a)x^{n-1} + \dots p_1(a)x +p_0(a)\}
$$
where $p_0(a) = \pm1$ is constant, and $p_{j_0}(t)$ is nonconstant for some $1\leq j_0\leq n-1$.
Assume that $K$ contains all roots of $f_a(x)$ for all natural numbers $a \geq N_0$.
 If $L$ is a totally imaginary quadratic extension of $K$, then the first-order theory of $\OO_L$ is undecidable.
   \end{theorem}
   
    \begin{proof} 
   By Lemma \ref{poly-lem}, it suffices to find a nonconstant polynomial $f$ and a definable subset $X_0\subseteq \OO_K$ such that $f(n) \in X_0$ for all sufficiently large $n \in \ZZ_{\geq N_0}$.
   
Let $N = \#\mu_L$, as in Lemma \ref{cm-lem}. We choose our polynomial to be 
$$
	f(x) = Q_{2Nk}(p_0(x), \dots, p_{n-1}(x)),
$$ where $Q_{2Nk}$ is defined in Lemma \ref{poly-sum} and $k$ is chosen according to Lemma \ref{nonconst-lem} so that $f(x)$ is nonconstant.  By definition, $f(a) = Q_{2N}(p_0(a), \dots, p_{n-1}(a))$ is the sum of $2N$-th powers of units of $\OO_K$ for each $a\geq N_0$, so we may define 
   $$
   	X_0 = \{\alpha_1^{2N} + \dots + \alpha_n^{2N} \mid \alpha_i\in \OO_L^\times\}.
$$ 
  This subset is definable in $\OO_L$, contains $f(n)$ for $n \geq N_0$ by assumption, and $X_0\subseteq \OO_K$ by Lemma \ref{cm-lem}.
    \end{proof}

By using the results discussed in the previous section, this theorem implies the our main theorem on the undecidability of totally imaginary fields.

\begin{theorem} 
 \label{full-thm}
 Let $K$ be an infinite totally real extension of $\QQ$ which is contained in $F^{(d)}_{ab}$ for some $d\geq 2$ and some number field $F$.  Assume $K$ contains all roots of a parametrized family of polynomials 
  $$
	\{f_a(x) = x^n + p_{n-1}(a)x^{n-1} + \dots p_1(a)x +p_0(a) \mid a\in \ZZ_{> N_0}\}
$$
where  each $p_i(t)\in \ZZ[t]$ is a polynomial, $p_0(t) = \pm1$ is constant and $p_j(t)$ is nonconstant for some $1\leq j\leq n-1$.  If $L$ is any  totally imaginary quadratic extension of $K$, then the first-order theory of  $L$ is undecidable.
  \end{theorem}
   \begin{proof} 
  Using Proposition \ref{north-jr} and Theorem \ref{northcott}, we see that $\JR(\OO_K) = \infty$.  By Theorem \ref{definability}, $\OO_L$ is definable in $L$. Thus the undecidability follows from Theorem \ref{main-thm}.
   \end{proof}

%%%
%
%%%
\section{Examples}
\label{examples}

We will now give some concrete examples of families of polynomials $\{f_a(x) : a\in \ZZ_{\geq N_0}\}$ which satisfy Theorem \ref{full-thm}.   In each case, $K$ can be taken to be the totally real field generated by all roots of the polynomials $\{f_a(x) \}$, or any extension thereof which is contained in $F^{(d)}_{ab}$ for some number field $F$ and some integer $d\geq 1$.  Then Theorem \ref{full-thm} implies that any totally imaginary quadratic extension $L$ of $K$ has undecidable first-order theory.

%%%
%
%%%
\subsection{Polynomials generating cyclic extensions of $\QQ$:}
\begin{enumerate}
	\item \textbf{The quadratic case:} Choosing $f_a(x) = x^2 - 2ax - 1$ produces the family of polynomials considered by Mart\'{\i}nez-Ranero, Utreras and Videla to prove the undecidability of $\QQ^{(2)}$ \cite[Lemma 7]{mruv} .
	More generally, one may use the family of polynomials $x^2 - p(a)x - 1$, where $p(t)$ is any nonconstant polynomial.  
	Using these polynomials also shows that the fields $\QQ^{(d)}_{ab}$ are undecidable for all $d \geq 2$, since each field is a totally imaginary quadratic extension of a totally real field.
	\item  \textbf{The cubic case:} Shanks describes some ``simplest cubic extensions" \cite{shanks} which are totally real and generated by roots of polynomials of the form 
	$$
		 x^3 - ax^2 - (a + 3)x - 1
	$$ for $a \geq -1$. Similarly, Kishi \cite{kishi} gives the family of polynomials
	$$
		x^3 - n(n^2 + n+3)(n^2+2)x^2 - (n^3 + 2n^2 + 3n + 3)x - 1
	$$
	for $n \in \ZZ$. Each polynomial in both families generates a totally real cyclic cubic extension of $\QQ$.
	\item  \textbf{The quartic case:} For $t\geq 4$, the following polynomials, constructed by Gras \cite[Proposition 6]{gras-quartic}, generate cyclic quartic totally real extensions of $\QQ$.
	$$
		x^4 - tx^3 -6x^2 +tx +1.
	$$
	
	\item  \textbf{The quintic case:} The following quintic polynomials, found by E. Lehmer, generate cyclic quintic totally real extensions of $\QQ$ for any $a\in \ZZ$; see the paper of Schoof and Washington \cite[\S3]{sw}.
	\begin{align*}
		x^5 + a^2x^4 &- (2a^3 +6a^2 + 10a + 10)x^3\\
			 &+ (a^4 + 5a^3 + 11a^2 + 15a + 5)x^2 + (a^3 + 4a^2 + 10a + 10) x + 1.
	\end{align*}
	
	\item  \textbf{The sextic case:} For $a\geq 7$, the following polynomials generate cyclic sextic totally real extensions of $\QQ$, as proved by Gras, \cite{Gras}.
	$$
		x^6 -2(a-1)x^5 -5(a+2)x^4 -20x^3 +5(a-1)x^2 + (2a + 4)x +1.
	$$
	
\end{enumerate}
	
\subsection{Polynomials generating non-abelian extensions of $\QQ$:}
\begin{enumerate}
	\item Let $F = \QQ(\sqrt{d})$ for some fixed square-free $d\in \NN$.  For any $a,b\in \ZZ$, the element $(a + b\sqrt{d})^2 + 1$ is totally positive, so $a + b\sqrt{d} + \sqrt{(a + b\sqrt{d})^2 + 1}$ is a totally real unit whose minimal polynomial is
	$$
		x^4 - 4ax^3 +(4(a^2 - b^2 d) - 2)x^2 + 4ax^2 +1
	$$
We therefore get an infinite 2-parameter family of suitable polynomials. Although the roots of such polynomials do not generally generate abelian extensions of $\QQ$, the roots lie in $F^{(2)} = F^{(2)}_{ab}$, so our theorem applies.

\item More generally, if $\theta$ is any totally real algebraic integer, then $u(\theta) = \theta + \sqrt{\theta^2+1}$ is a totally real unit which satisfies $x^2 - 2\theta x -1$. 
Let $\alpha$ be a fixed totally real algebraic integer with conjugates $\{\alpha = \alpha_1, \dots, \alpha_n\}$ and let $F$ be a number field containing $\alpha$, enlarged to be Galois without loss of generality.  Let $h(t_1, t_2)\in \ZZ[t_1, t_2]$ be a polynomial satisfying $\deg_{t_1}(h) > 0$ and $\deg_{t_2}(h) = [\QQ(\alpha) : \QQ] -1$.
  Define $\theta(a) = h(a, \alpha)$.  Then $\theta(a)$ is totally real for all $a\in \ZZ$ and we can take 
$$
	f_a(x)
		=\prod_{i = 1}^n  (x^2 - 2h(a, \alpha_i)x -1)
$$
which has $u(\theta(a))$ as a root by design.  The coefficients of $f_a(x)$ are polynomials in $a$ which depend on $\alpha$, and at least one is nonconstant because the degree restrictions on $h(t_1,t_2)$ ensures that $h(a,\alpha)$ outputs infinitely many values as $a$ varies.  Again, the roots of all the polynomials $f_a(x)$ lie in $F^{(2)} = F^{(2)}_{ab}$, so our theorem applies, although it will not generally be contained in an abelian extension of $\QQ$.
\end{enumerate}¥

\end{document}